\documentclass[12pt]{article}
\usepackage{amsthm}
\usepackage{amsmath}
\usepackage{cite}
\usepackage{tikz}
\usetikzlibrary{arrows}
\usepackage{amssymb}
\usepackage{enumerate}
\usepackage[margin=1.2in]{geometry}

\usepackage{float}

\newtheorem{theorem}{Theorem}[section]

\newtheorem{proposition}[theorem]{Proposition}

\newtheorem{lemma}[theorem]{Lemma}

\newtheorem{question}[theorem]{Question}

\include{epsf}

\begin{document}

\title{Hamiltonian circle actions on complete intersections}

\author{Nicholas Lindsay}

\maketitle

\begin{abstract}
We study the problem of determining which diffeomorphism classes of K\"{a}hler manifolds admit a Hamiltonian circle action. Our main result is the following: Let $M$ be a closed symplectic manifold, diffeomorphic to a complete intersection with complex dimension $4k$, having a Hamiltonian circle action such that each component of the fixed point set is an isolated fixed point or has dimension $2 \mod 4$.  Then $M$ is diffeomorphic to  $\mathbb{CP}^{4k}$, a quadric $Q \subset \mathbb{CP}^{4k+1}$ or an intersection of two quadrics $Q_1 \cap Q_2 \subset \mathbb{CP}^{4k+2}$. 
\end{abstract}
\section{Introduction}

Smooth complete intersections over $\mathbb{C}$ having infinite automorphism groups are classified. Benoist has shown that aside from cubic plane curves,  projective spaces and quadrics are the only examples \cite{B}[Theorem 3.1]. Here we prove a symplectic version of (a partial subcase) of that theorem, our main result is as follows \footnote{Throughout the article, dimension refers to real dimension unless stated otherwise.}.

 \begin{theorem} \label{main} Let $(M,\omega)$ be a closed symplectic manifold with dimension $8k$, diffeomorphic to a complete intersection. Suppose that $M$ has a Hamiltonian $S^1$-action such that each component of the fixed point set is an isolated fixed point or has dimension $2 \mod 4$. Then $M$ is diffeomorphic to $\mathbb{CP}^{4k}$, a quadric $Q \subset \mathbb{CP}^{4k+1}$ or an intersection of two quadrics $Q_1 \cap Q_2 \subset \mathbb{CP}^{4k+2}$. 

\end{theorem} 

The fixed point set of a Hamiltonian circle action is a collection of symplectic submanifolds of possibly different dimensions. We emphasise that the condition on the fixed point set is simultaneous, i.e. if the fixed point set consisted of an isolated fixed point and a symplectic $6$-manifold, then Theorem \ref{main} would apply. In particular, Theorem \ref{main}  applies to actions such that all components of the fixed point set have dimension at most $2$.

 Our motivation for proving Theorem \ref{main} comes in two parts:

\begin{enumerate}

\item It recovers part of the algebraic result \cite{B}[Theorem 3.1] using purely topological methods. \item  It excludes the existence of exotic Hamiltonian circle actions on a large class of Fano varieties. \end{enumerate}

Exotic Hamiltonian circle actions on K\"{a}hler manifolds are known to exist in various senses  \cite{To1,GKZ1,GKZ2,LP2}, although conjecturally do not occur on Fano varieties.  In \cite{DW} it was shown that complete intersections with dimension $6$ have a smooth circle action $\iff$ the holomorphic automorphism group is infinite.  

As was pointed out in \cite[Page 3]{DW} it is also possible to say something about smooth actions in higher dimensions. Let $X$ be a complete intersection of complex dimension $2k$, then $$\hat{A}(X) = 0 \iff X\text{ is Fano} .$$ Hence if the complete intersection is even complex dimension, spin and non-Fano the non-vanishing of the $\hat{A}$-genus rules out the existence of any smooth circle action, by a theorem of Atiyah and Hirzebruch \cite{AH}.

We give a brief overview of the proof of Theorem \ref{main}. The first step is to prove Proposition \ref{gen}, which may be considered as a generalisation of a formula of Jones and Rawnsley \cite[Theorem 1.1]{JR}. Proposition \ref{gen} expresses the signature of a symplectic manifold with a Hamiltonian circle action as the alternating sum of its even degree Betti numbers, under certain assumptions on the Betti numbers and the fixed point set. 
   
 The next step of the proof is Lemma \ref{can}, where it is shown that for complete intersections of complex dimension $4k$ the alternating sum of the even degree Betti numbers is equal to the middle Betti number. Hence, if the underlying smooth manifold has a Hamiltonian circle action then the intersection form on the middle cohomology group is positive definite. Finally, we apply a classification of complete intersections whose intersection form is positive definite due to Libgober and Wood \cite{LW}, which completes the proof of Theorem \ref{main}. 
 
 \textbf{Acknowledgements}. I would like to thank Volker Puppe for pointing me to the reference \cite{DW} which initiated my interest in this question. I would like to thank the referee for helpful comments.

 \section{Preliminaries}
 \subsection{Hamiltonian circle actions}

 We recall some foundational results on Hamiltonian circle actions, the proofs of which may be found in \cite[Chapter 5]{MS}. Suppose a closed symplectic manifold $(M,\omega)$ has an $S^1$-action generated by a vector field $X$, then the circle action is called Hamiltonian if $dH= \omega(X,\cdot)$ for some function $H$ on $M$ called the Hamiltonian. The fixed point set of the action, denoted $M^{S^1}$  is equal to the critical set of $H$, and is a collection of symplectic submanifolds. It will also be convenient to define a subset $M^{S^1}_{I} \subset M^{S^1}$ consisting of isolated fixed points. 
 
 Also, $H$ is a perfect Morse-Bott function.  For a fixed component $F$ denote by $\lambda_{F}$  half the Morse-Bott index of $H$ along $F$. It holds that  $\lambda_{F}$ is also equal to the the number of strictly negative weights of the $S^1$-action along $F$ counted with multiplicity. Together, these facts yield the following localisation formula for the Betti numbers \cite[Page 9]{LT}.
 
 \begin{theorem} \label{locbet}
 Let $(M,\omega)$ be a closed symplectic manifold having a Hamiltonian circle action. Then for each $i$, $$ b_{i}(M) = \sum_{F \subset M^{S^1}} b_{i - 2\lambda_{F}(M)} (F) ,$$ where the sum runs over connected components of the fixed point set. 
 \end{theorem}
 
 The following formula is a consequence of the equivariant version of the Atiyah-Singer index theorem. The formula we will use is derived in \cite[Section 5.8]{HBJ}. We use $\sigma(M)$ to denote the signature.
  
  \begin{theorem} \label{locsig}
 Let $(M,\omega)$ be a closed symplectic manifold having a Hamiltonian circle action. Suppose that the fixed point set contains only isolated fixed points or submanifolds of dimension $2 \mod 4$. Then $$\sigma(M) = \sum_{p \in M^{S^1}_{I}} (-1)^{\lambda_{p}} .$$ 
 \end{theorem}
 \begin{proof}
 This formula follows from the formula of \cite[Section 5.8]{HBJ}, expressing the signature as a sum over the fixed point components. Fixed submanifolds of dimension $2 \mod 4$ contribute $0$ to the formula  (see the proof of \cite[Theorem 4.2]{JR}). Hence, the sum simplifies as a sum over $M^{S^1}_{I}$, the contribution of an isolated fixed point $p$ in terms of $\lambda_{p}$ is derived in \cite[Section 3]{JR}.
 \end{proof}

\subsection{Complete intersections}
Next we turn our attention to complete intersections, stating two preliminary results. For a comprehensive introduction to the topology of complete intersections see the introduction of \cite{LW}.  Firstly we state a well known consequence of the Lefschetz hyperplane theorem. 

\begin{lemma} \label{Lef}
Let $M$ be a smooth complete intersection of dimension $8k$, then $b_{i}(M) = b_{i}(\mathbb{CP}^{4k})$ for $i \neq 4k$.
\end{lemma}

 Finally we state the result of Libgober and Wood which we use \cite[Page 638]{LW}.

\begin{theorem} \cite{LW} \label{libwood}
Suppose that $X$ is a complete intersection of dimension  $8k$, and suppose that intersection form on $H^{4k}(X,\mathbb{Z})$ is positive definite. Then, $X$ is diffeomorphic to a projective space, a quadric or an intersection of two quadrics.
\end{theorem}

\section{Proof of Theorem \ref{main}}

We now proceed to the proof of our main result Theorem \ref{main}. Before giving the details of the proof, we describe the general approach. The idea of the proof was motivated by a theorem of Jones and Rawnsley  \cite[Theorem 1.1]{JR}, which states that if a closed symplectic $(M,\omega)$ has a Hamiltonian circle action with isolated fixed points then the following formula holds $$ \sigma(M) = \sum_{j \in \mathbb{Z}}b_{4j}(M) - \sum_{j \in \mathbb{Z}} b_{4j +2}(M) .$$

The first step in the proof of Theorem \ref{main} is Proposition \ref{gen}. It is a generalisation of \cite[Theorem 1.1]{JR} which is particularly well suited for studying complete intersections, since it assumes that the Betti numbers satisfy a condition that complete intersections satisfy (see Lemma \ref{Lef}). 

The method of the proof is to use the localisation formula for the Betti numbers Theorem \ref{locbet} and the localisation formula for the signature Theorem \ref{locsig}. Using Theorem \ref{locbet}, we express the alternating sum of the even degree Betti numbers as a certain sum depending on the fixed point set. Next, using a claim about the Betti numbers of the fixed components which is delayed to the end of the proof (Claim A), we show that this simplifies to a sum over the isolated fixed points. Finally we show that this expression is equal to the formula for the signature from Theorem \ref{locsig}. This is Jones and Rawnsley's original method of proving \cite[Theorem 1.1]{JR}, with additional considerations needed to deal with fixed components of dimension $2 \mod 4$.

\begin{proposition} \label{gen}
Suppose that $M$ is a closed symplectic manifold of dimension $8k$. Suppose that $M$ has a Hamiltonian $S^1$-action such that each component of the fixed point set is an isolated fixed point or has dimension $2 \mod 4$. Suppose also that $b_{i}(M) = b_{i}(\mathbb{CP}^{4k})$ for $i \neq 4k$. Then $$ \sigma(M) = \sum_{j \in \mathbb{Z}}b_{4j}(M) - \sum_{j \in \mathbb{Z}} b_{4j +2}(M) .$$

\end{proposition}
\begin{proof} We assume that $M$ satisfies the assumptions of the theorem. Let $M^{S^1}$ denote the fixed point set of the action, and let $M^{S^1}_{I} \subset M^{S^1}$ denote the subset consisting of isolated fixed points. Recall that for a connected component $N \subset M^{S^1}$, $\lambda_{N} \in \mathbb{Z}$ denotes half the Morse-Bott index of $N$ with respect to the Hamiltonian.

\textbf{Claim A.} For any component $N \subset M^{S^1}$ with dimension $2 \mod 4$, $$ \sum_{j \in \mathbb{Z}}b_{4j - 2\lambda_N}(N) - \sum_{j \in \mathbb{Z}} b_{4j +2 - 2\lambda_N}(N) = 0. $$

We will first show how the proposition follows from Claim A. We consider the quantity in question $$\alpha =  \sum_{j \in \mathbb{Z}} (b_{4j}(M) - b_{4j +2}(M)) , $$ substituting the formula from Theorem \ref{locbet} implies that $$ \alpha = \sum_{j \in \mathbb{Z} , \: F \subset M^{S^1}} (b_{4j - 2\lambda_F}(F) - b_{4j +2 - 2\lambda_{F}}(F)).   $$
 Next, applying Claim A gives $$ \alpha = \sum_{j \in \mathbb{Z} , \: p \in M^{S^1}_I} (b_{4j - 2\lambda_p}(p) - b_{4j +2 - 2\lambda_{p}}(p))   . $$ Substituting the Betti numbers of a point gives
 
 $$\alpha = |\{p \in M^{S^1}_I: \lambda_{p} \text{ is even} \}| - |\{p \in M^{S^1}_I : \lambda_{p} \text{ is odd} \}|. $$
 
The right hand side of this expression is equal to the expression for $\sigma(M)$ given in Theorem \ref{locbet}. Hence to prove the proposition, it remains to prove Claim A.
 
 \textbf{Proof of Claim A.} Let $N$ be a fixed submanifold with dimension $2 \mod 4$. We will first show that
 
 $$ \sum_{j \in \mathbb{Z}}b_{4j}(N) - \sum_{j \in \mathbb{Z}} b_{4j +2 }(N) = 0. $$
 
 By assumption, $\dim(N) = 2m$ where $m$ is odd. Furthermore since $N$ is a symplectic submanifold, all of the even degree Betti numbers of $N$ are positive. By Theorem \ref{locbet} and Lemma \ref{Lef}, it follows that $N$ has at most one Betti number which is at least $2$. On the other hand, since $m$ is odd and by Poincar\'{e} duality, all of the integers appearing as even degree Betti numbers of $N$ appear in at least two degrees. Hence, by the above all of the even degree Betti numbers of $N$ are equal to $1$. It follows since $m$ is odd that 
 
 $$ \sum_{j \in \mathbb{Z}}b_{4j}(N) - \sum_{j \in \mathbb{Z}} b_{4j +2 }(N) =  1 -1 +  \ldots +1-1 = 0. $$ Finally, the above equation implies Claim A:  $$ \sum_{j \in \mathbb{Z}}b_{4j - 2\lambda_N}(N) - \sum_{j \in \mathbb{Z}} b_{4j +2 - 2\lambda_N}(N) = (-1)^{\lambda_N} (\sum_{j \in \mathbb{Z}}b_{4j}(N) - \sum_{j \in \mathbb{Z}} b_{4j +2 }(N) )= 0. $$

\end{proof}

The next step in the proof of Theorem \ref{main} is an elementary lemma showing that the alternating sum of even degree Betti numbers of a complete intersection of dimension $8k$ is equal to the middle Betti number.

\begin{lemma} \label{can}
Suppose that $M$ is a closed manifold of dimension $8k$ satisfying $b_{i}(M) = b_{i}(\mathbb{CP}^{4k})$ for $i \neq 4k$. Then $$ \sum_{j \in \mathbb{Z}}b_{4j}(M) - \sum_{j \in \mathbb{Z}} b_{4j +2}(M)  = b_{4k}(M). $$
\end{lemma}
\begin{proof}
This follows from a direct computation. The alternating sum simplifies as follows: $$ \sum_{j \in \mathbb{Z}}b_{4j}(M) - \sum_{j \in \mathbb{Z}} b_{4j +2}(M)  =  1 -1 \ldots +1-1 + b_{4k}(M) -1+1 \ldots -1 +1 = b_{4k}(M). $$ 
\end{proof}

The proof of Theorem \ref{main} now follows from combining what we have already proven with a theorem of Libgober and Wood, Theorem \ref{libwood}.

\begin{proof}[Proof of Theorem \ref{main}]  Suppose that $M$ satisfies the assumptions of the theorem. By Proposition \ref{gen} 
$$ \sigma(M) = \sum_{j \in \mathbb{Z}}b_{4j}(M) - \sum_{j \in \mathbb{Z}} b_{4j +2}(M) .$$ Moreover, by Lemma \ref{Lef}, $b_{i}(M) = b_{i}(\mathbb{CP}^{4k})$ for $i \neq 4k$. By Lemma \ref{can} $$b_{4k}(M) =  \sum_{j \in \mathbb{Z}}b_{4j}(M) - \sum_{j \in \mathbb{Z}} b_{4j +2}(M)  = \sigma(M). $$ So, since the signature of the intersection form of the middle cohomology is equal to its rank, the intersection form on $H^{4k}(M,\mathbb{R})$ is positive definite. Now the theorem follows from Theorem \ref{libwood}.
\end{proof}

\subsection{Discussion}

\subsubsection{The case when $\dim(M) = 8k+4$.} We remark that if a closed symplectic manifold has dimension $8k+4$ and has a Hamiltonian $S^1$-action satisfying the same condition as Proposition \ref{gen} then identical considerations to Proposition \ref{gen}  and Lemma \ref{can} show that $\sigma(M) = 2- b_{4k+2}(M)$. This may be rephrased as $b_{4k+2}^{+}(M)=1$. A quadric $Q$ with complex dimension $4k+2$ satisfies $b_{4k+2}(Q)=2$ and  $\sigma(Q) = 0$. So the intersection form of $Q$ is the standard hyperbolic form $\tau(x,y) = x^2-y^2$, in a suitable basis of $H^{4k+2}(Q,\mathbb{R})$. In particular, intersection form  is  in general  not positive definite in dimension $8k+4$. 

\subsubsection{ The intersection of two quadrics.} We note that unfortunately Theorem \ref{main} leaves one unsettled case in each dimension: the intersections of two quadrics. Such varieties do not have an algebraic torus action due to \cite{B}[Theorem 3.1]. We leave the following as a question:
\begin{question}
Does the underlying smooth manifold of  an intersection of $2$ quadrics with complex dimension at least $4$  have a symplectic form exhibiting a Hamiltonian circle action?
\end{question}

We note that the underlying smooth manifold of $Q_1 \cap Q_2 \subset \mathbb{CP}^4$ has a Hamiltonian circle action since it diffeomorphic to a blow up of $\mathbb{CP}^2$ in five points. But it is unclear whether this is a low-dimensional coincidence.

Institute of Mathematical Sciences, ShanghaiTech University No. 393 Huaxia Middle Road, Pudong New Area, Shanghai,  China. Email: lindsaynj@shanghaitech.edu.cn
\end{document}